\documentclass[10pt]{amsart}
\usepackage[english]{babel}
\usepackage{amssymb}
\usepackage{amsmath}
\usepackage{lscape}

\newtheorem{dummy}{Dummy}

\newtheorem{lemma}[dummy]{Lemma}
\newtheorem{theorem}[dummy]{Theorem}
\newtheorem{proposition}[dummy]{Proposition}
\newtheorem{corollary}[dummy]{Corollary}

\theoremstyle{definition}

\newtheorem{example}[dummy]{Example}

\newtheorem{remark}[dummy]{Remark}

\newcommand{\ignore}[1]{}

\author{C. Brown}
\author{S. Pumpl\"un}

\email{christian\_jb@hotmail.co.uk; susanne.pumpluen@nottingham.ac.uk}
\address{School of Mathematical Sciences\\
University of Nottingham\\ University Park\\ Nottingham NG7 2RD\\
United Kingdom }



\begin{document}

\title[Skew polynomials and nonassociative algebras]
{How a nonassociative algebra reflects the properties of a skew polynomial}

\begin{abstract}
Let $S$ be  a unital associative division ring and $S[t;\sigma,\delta]$ be a
skew polynomial ring, where $\sigma$ is an endomorphism of $S$ and $\delta$ a left $\sigma$-derivation.
 For each $f\in S[t;\sigma,\delta]$ of degree
$m>1$ with a unit as leading coefficient, there exists a unital nonassociative algebra whose behaviour reflects the properties of
$f$.
These algebras yield canonical examples of right division algebras when $f$ is irreducible.
The structure of their right
nucleus depends on the choice of $f$. In the classical literature, this nucleus appears  as the
 eigenspace of $f$, and is used to investigate the irreducible factors of $f$.
  We give necessary and sufficient criteria for skew polynomials of low degree to be irreducible. These yield examples of new division algebras $S_f$.
\end{abstract}

\maketitle

%
\section*{Introduction}
%

 The investigation of skew polynomials is an active area in algebra which  has
applications to coding theory, to solving differential and difference
equations, and in engineering, to name just a few. For instance, linear differential operators (where $\sigma=id$) and linear difference operators (where $\delta=0$) are special cases of skew polynomials.

Let $D$ be a unital associative division ring and $R=D[t;\sigma,\delta]$ a skew
polynomial ring, where $\sigma$ is an endomorphism of $D$ and $\delta$ a left $\sigma$-derivation.
Suppose $f\in D[t;\sigma,\delta]$ has degree $m$. Using right division by $f$ to
define a multiplication on the set of skew polynomials of degree less
than $m$, this set becomes a unital nonassociative algebra we denote by $S_f$. The
algebra $S_f$ generalizes the classical quotient algebra construction
 when factoring out a two-sided ideal generated by a right invariant skew polynomial $f$.
 When choosing $f$ and $R$ in the right way, it can be also seen as a generalization of certain crossed product algebras and some Azumaya algebra constructions.
First results on the structure of the algebras $S_f$ which initially were defined by Petit in \cite{P66} have appeared in \cite{P66, P68, BP18.1, BP18.2, P15, P15.2,  P18.3}.
First applications to coding theory have appeared for instance in \cite{P17, P18.1, P18.2}.

 Recently, a computational criterion for deciding whether a bounded skew polynomial is irreducible was developed
 in \cite{GLN}.
 The method heavily relies on being able to find the zero divisors in the right nucleus
 of $S_f$ (although the simple algebra employed there, called the eigenspace of $f$, is not recognized as the right nucleus of $S_f$ in that paper). The method is only applicable for certain set-ups when the input data $S$, $\sigma$ and $\delta$ are effective and computable,
  but it demonstrates the importance of developing a better understanding of the algebras $S_f$ and their algebraic structure.
Independently, effective algorithms to compute the eigenspace (and thus the right nucleus of $S_f$, again not recognized as such) for the special case that $R=\mathbb{F}_q(x)[t;\sigma,\delta]$ can be found in
\cite{GZ}, and for $R=\mathbb{F}_q[t;\sigma]$  in \cite{G0}, \cite{R}. In all cases the eigenspace is a crucial tool to understand the decomposability of the skew polynomial $f$.

This paper consists of two parts. The first one considers the structure of the right nucleus of the algebras $S_f$, establishing how it reflects the type of the skew polynomial $f$ it is defined with,
but also the important role irreducible polynomials play in the construction of classes of nonassociative unital (right) division algebras.

The second part looks at skew
polynomials of low degree as well as the polynomial $f(t)=t^m-a$, and when these polynomials are irreducible in
$D[t;\sigma,\delta]$,
in order to obtain examples for the construction of (right) division algebras.

After establishing the basic terminology in Section \ref{sec:prel}, we define Petit algebras in Section \ref{subsec:structure}
and collect some results on their right nuclei in Section \ref{section:Right Semi-Invariant Polynomials}.
We investigate when the algebras $S_f$ are right (and not left) division
algebras in Section \ref{section:When is S_f a Division Algebra?}. A
necessary condition for $S_f$ being a right division algebra is that the polynomial $f$ is irreducible.
 We then collect some irreducibility criteria for polynomials of
low degree and the polynomial $f(t)=t^m-a$ in  both $R = D[t;\sigma]$ and $R = D[t;\sigma,\delta]$ in Sections
\ref{section:Irreducibility Criteria in D[t;sigma]} and \ref{section:Irreducibility Criteria in D[t;sigma,delta]},
including the special case where $D$ is a finite field.

We point out that there exists some kind of Eisenstein valuation criteria to test a skew polynomial over a division ring for reducibility,
using some (noncommutative) valuation theory for skew polynomial rings \cite{CZ, GMR}.
We believe our criteria  are more tractable for the types of skew polynomials we consider.

Most of this work is part of the first author's PhD thesis \cite{CB} written under the supervision of the second author.

%
%

\section{Preliminaries} \label{sec:prel}

\subsection{Skew polynomial rings}

Let $S$ be a unital associative ring, $\sigma$ a ring endomorphism of
$S$ and $\delta:S\rightarrow S$ a \emph{left $\sigma$-derivation},
i.e. an additive map such that
$\delta(ab)=\sigma(a)\delta(b)+\delta(a)b$
for all $a,b\in S$.
Then the \emph{skew polynomial ring} $R=S[t;\sigma,\delta]$ is the
set of skew polynomials $g(t)=a_0+a_1t+\dots +a_nt^n$ with $a_i\in
S$, with term-wise addition and where the multiplication is defined
via $ta=\sigma(a)t+\delta(a)$ for all $a\in S$ \cite{O1}. That means,
$$at^nbt^m=\sum_{j=0}^n a(\Delta_{n,j}\,b)t^{m+j}$$ for all $a,b\in
S$, where the map $\Delta_{n,j}$ is defined recursively via
$$\Delta_{n,j}=\delta(\Delta_{n-1,j})+\sigma (\Delta_{n-1,j-1}),$$ with
$\Delta_{0,0}=id_S$, $\Delta_{1,0}=\delta$, $\Delta_{1,1}=\sigma $.
Therefore $\Delta_{n,j}$ is the sum of all monomials in $\sigma$
and $\delta$ of degree $j$ in $\sigma$ and degree $n-j$ in $\delta$
\cite[p.~2]{J96}. If $\delta=0$, then $\Delta_{n,n}=\sigma^n$.

For $\sigma=id$ and $\delta=0$, we obtain the usual ring of left
polynomials $S[t]=S[t;id,0]$. Define ${\rm Fix}(\sigma)=\{a\in S\,|\,
\sigma(a)=a\}$ and ${\rm Const}(\delta)=\{a\in S\,|\, \delta(a)=0\}$.

 For $f(t)=a_0+a_1t+\dots +a_nt^n\in R$ with $a_n\not=0$ define ${\rm deg}(f)=n$ and ${\rm deg}(0)=-\infty$.
Then ${\rm deg}(gh)\leq{\rm deg} (g)+{\rm deg}(h)$ (with equality if
$h$ has an
 invertible leading coefficient, or $g$ has an
 invertible leading coefficient and $\sigma$ is injective, or if $S$ is a division ring).
 An element $f\in R$ is \emph{irreducible} in $R$ if it is not a unit and  it has no proper factors, i.e if there do not exist $g,h\in R$ with
 ${\rm deg}(g),{\rm deg} (h)<{\rm deg}(f)$ such
 that $f=gh$.

\subsection{Nonassociative algebras} \label{subsec:2}

Let $R$ be a unital commutative ring and let $A$ be an $R$-module. We
call $A$ an \emph{algebra} over $R$ if there exists an $R$-bilinear
map $A\times A\mapsto A$, $(x,y) \mapsto x \cdot y$, usually denoted simply
by juxtaposition $xy$, the  \emph{multiplication} of $A$. An algebra
$A$ is called \emph{unital} if there is an element in $A$, denoted by
1, such that $1x=x1=x$ for all $x\in A$. We will only consider unital
algebras.

For an $R$-algebra $A$, associativity in $A$ is measured by the {\it
associator} $[x, y, z] = (xy) z - x (yz)$. The {\it left nucleus} of
$A$ is defined as ${\rm Nuc}_l(A) = \{ x \in A \, \vert \, [x, A, A]
= 0 \}$, the {\it middle nucleus}  as ${\rm Nuc}_m(A) = \{ x \in A \,
\vert \, [A, x, A]  = 0 \}$ and  the {\it right nucleus}  as ${\rm
Nuc}_r(A) = \{ x \in A \, \vert \, [A,A, x]  = 0 \}$. ${\rm Nuc}_l(A)$, ${\rm Nuc}_m(A)$ and ${\rm Nuc}_r(A)$ are associative
subalgebras of $A$. Their intersection
 ${\rm Nuc}(A) = \{ x \in A \, \vert \, [x, A, A] = [A, x, A] = [A,A, x] = 0 \}$ is the {\it nucleus} of $A$.
${\rm Nuc}(A)$ is an associative subalgebra of $A$ containing $R1$
and $x(yz) = (xy) z$ whenever one of the elements $x, y, z$ is in
${\rm Nuc}(A)$. The  {\it commuter} of $A$ is defined as ${\rm
Comm}(A)=\{x\in A\,|\,xy=yx \text{ for all }y\in A\}$ and the {\it
center} of $A$ is ${\rm C}(A)=\text{Nuc}(A)\cap  {\rm Comm}(A)$
 \cite{Sch}.

A nonassociative ring $A\not=0$ (resp., an algebra $A\not=0$ over a
field $F$) is called a \emph{left division ring (resp., algebra)}, if
for all $a\in A$, $a\not=0$, the left multiplication  with $a$,
$L_a(x)=ax$, is a bijective map, and  a \emph{right division ring (resp.,
algebra)}, if for all $a\in A$, $a\not=0$, the right multiplication
with $a$, $R_a(x)=xa$, is a bijective map.
  An algebra $A\not=0$ over a field $F$ is called a {\it division algebra} if for all $a\in A$, $a\not=0$, both
the left and right multiplication with $a$ are bijective. A division
algebra $A$ does not have zero divisors. If $A$ is a
finite-dimensional algebra over $F$, then $A$ is a division algebra
over $F$ if and only if $A$ has no zero divisors \cite{Sch}.

A nonassociative ring $A\not=0$ has no zero divisors if and only if
$R_a$ and $L_a$ are injective for all $0\not=a\in A$.

Note that every algebra $A$ is a right ${\rm Nuc}_r(A)$-module and the left
multiplication $L_a$ is ${\rm Nuc}_r(A)$-linear for all $0\not=a\in A$.

%
%
%
\section{Nonassociative algebras obtained from skew polynomials} \label{subsec:structure}

Let $S$ be a unital associative ring and $S[t;\sigma,\delta]$ a skew polynomial
ring where $\sigma$ is injective.

\subsection{}
Assume $f(t)=\sum_{i=0}^{m}a_it^i \in R=S[t;\sigma,\delta]$ has an invertible leading coefficient $a_m\in S^\times$. Then
for all $g(t)\in R$ of degree $l\geq m$,  there exist  uniquely
determined $r(t),q(t)\in R$ with
 ${\rm deg}(r)<{\rm deg}(f)$, such that
$g(t)=q(t)f(t)+r(t),$ and if $\sigma\in{\rm Aut}(D)$, also  uniquely
determined $r(t),q(t)\in R$ with ${\rm deg}(r)<{\rm deg}(f)$, such
that $g(t)=f(t)q(t)+r(t)$ (\cite{CB},\cite[Proposition 1]{P15}).

Let ${\rm mod}_r f$ denote the remainder of right division by $f$
and ${\rm mod}_l f$ the remainder of left division by $f$.
The skew polynomials of degree less that $m$ canonically represent the elements of the (left resp. right)
$S[t;\sigma,\delta]$-modules $S[t;\sigma,\delta]/
S[t;\sigma,\delta]f$ and $S[t;\sigma,\delta]/ fS[t;\sigma,\delta]$.
Moreover, $$R_m=\{g\in S[t;\sigma,\delta]\,|\, {\rm deg}(g)<m\}$$
 together with the multiplication
   \[g\circ h=
  \begin{cases}
  gh  \text{ if } {\rm deg} (g)+{\rm deg} (h) < m,\\
  gh \,\,{\rm mod}_r f \text{ if } {\rm deg} (g)+{\rm deg} (h) \geq m,
   \end{cases}
  \]
is a unital nonassociative ring $S_f=(R_m,\circ)$ also denoted by
$R/Rf$.
\\  If $\sigma\in{\rm Aut}(S)$, then $R_m$  together with
   \[g\circ h=
  \begin{cases}
  gh  \text{ if } {\rm deg} (g)+{\rm deg} (h) < m,\\
  gh \,\,{\rm mod}_l f \text{ if } {\rm deg} (g)+{\rm deg} (h) \geq m,
   \end{cases}
  \]
is a unital nonassociative ring $\,_fS=(R_m,\circ)$ also denoted by
$R/fR$. When the context is clear, we will drop the $\circ$ notation and simply use juxtaposition for multiplication in $S_f$.

$S_f$ and $\,_fS$ are unital nonassociative algebras over the
commutative subring $$S_0=\{a\in S\,|\, ah=ha \text{ for all } h\in
S_f\}={\rm Comm}(S_f)\cap S$$
 of $S$.
 Note that
$$ C(S)\cap{\rm Fix}(\sigma)\cap {\rm Const}(\delta)\subset S_0.$$
 For all invertible $a\in S$ we have $S_f = S_{af}$, so that without loss of generality it suffices
to only consider monic polynomials in the construction.
  If $f$ has degree 1 then $S_f\cong S$.  If $f$ is reducible then $S_f$ contains zero divisors.

 In the following, we assume $m\geq 2$ and call the algebras $S_f$  \emph{Petit algebras}
as the construction goes back to Petit \cite{P66,P68} (who only
considered division rings $S$). We will  focus on the  algebras
$S_f$, since the algebras $\,_fS$ are anti-isomorphic to Petit algebras \cite[Proposition 3]{P15}.

Note that for $0\not=a\in S_f$, left multiplication $L_a$ is an
$S_0$-module endomorphism. Moreover,  $R_a$ is a left $S$-module homomorphism for $0 \neq a \in S_f$.

Let $f\in S[t;\sigma,\delta]$ have degree $m\geq 2$ and an
invertible leading coefficient. Then $S_f$ is a free left $S$-module  of rank $m$ with basis $t^0=1,t,\dots,t^{m-1}$.
 $S_f$ is associative if and only if $Rf$ is a two-sided ideal in $R$.
 If $S_f$ is not associative then
$S\subset{\rm Nuc}_l(S_f),\,\,S\subset{\rm Nuc}_m(S_f)$ and $$\{g\in
R\,|\, {\rm deg}(g)<m \text{ and }fg\in Rf\}= {\rm Nuc}_r(S_f).$$
When $S$ is a division ring, these inclusions become equalities.
We have $t\in {\rm Nuc}_r(S_f)$, if and only if the powers of $t$ are
associative, if and only if $t^mt=tt^m$ in $S_f$.
 If $S$ is a division ring and  $S_f$ is not associative then $C(S_f)=S_0.$
 Let $f(t)=\sum_{i=0}^{m}a_it^i\in S[t;\sigma]$ with $a_0$ invertible.
If the endomorphism  $L_t$ (i.e.  left multiplication by $t$)  is
surjective then $\sigma$ is surjective.
  In particular, if $S$ is a division ring and $f$ irreducible, then $L_t$ surjective implies $\sigma$ surjective.
Moreover, if $\sigma$ is bijective then $L_t$ is surjective \cite[Theorem 4]{P15}.

Note that $C(S_f)={\rm Comm}(S_f)\cap {\rm Nuc}_l(S_f)\cap {\rm Nuc}_m(S_f)\cap {\rm Nuc}_r(S_f)$ and
so
$$S_0=\{a\in S\,|\, ah=ha \text{ for all } h\in S_f\}={\rm Comm}(S_f)\cap S\subset C(S_f).$$
If ${\rm Nuc}_l(S_f)= {\rm Nuc}_m(S_f)=S$ this yields that the center
$C(S_f)={\rm Comm}(S_f)\cap S\cap {\rm Nuc}_r(S_f)={\rm Comm}(S_f)\cap S$  of $S_f$ is identical to the ring
$S_0.$

%
%

\section{The right nucleus of $S_f$} \label{section:Right Semi-Invariant Polynomials}

In this Section, let $D$ be a division algebra with center $F$, $R=D[t;\sigma,\delta]$
 with $\sigma$ any endomorphism of $D$ and $\delta$ any left $\sigma$-derivation. Let $f \in R = D[t;\sigma,\delta]$
 be monic of degree $m\geq2$.

The largest subalgebra of
$R=D[t;\sigma,\delta]$ in which $Rf$ is a two-sided ideal is the \emph{idealizer}  $I(f)=\{g\in R\,|\, fg\in Rf\}$ of $Rf$. The
\emph{eigenring} of $f$ is then defined as the quotient $E(f)=I(f)/Rf=\{g\in R\,|\, {\rm deg}(g)<m \text{ and } fg\in Rf\}$.
This is also the right nucleus of the algebra $S_f$ \cite[Theorem 4]{P15}.

 \subsection{Some general observations}  The right nucleus is important when finding right factors of $f$;
 if $ {\rm Nuc}_r(S_f)$ contains zero divisors then $f$ is reducible \cite{P66}.
If $u,v\in {\rm Nuc}_r(S_f)$ are non-zero such that $uv=0$, then
the greatest common right divisor $gcrd(f,u)$  is a non-trivial right
factor of $f$, see e.g.  \cite{P15}. This was employed for instance in \cite{GLN}.

 Moreover, if $ft\in Rf$ then $t\in {\rm Nuc}_r(S_f)$, hence the powers of $t$ are associative in $S_f$. This in turn implies $t^mt=tt^m$  \cite[Theorem
5]{P15}. Moreover, $ft\in Rf$  if and only if $t\in {\rm Nuc}_r(S_f)$, if and only if the powers of $t$ are
associative, if and only if $t^mt=tt^m$ \cite{P66}. This yields:

\begin{lemma}
Let  $f\in D[t;\sigma,\delta]$. If $t\not\in {\rm Nuc}_r(S_f)$ or $f\in Rt $
 then $S_f$ does not have any associative subalgebra that contains all powers of $t$.

 In particular, if  $f$ is irreducible and  $t\not\in {\rm Nuc}_r(S_f)$,
 then $S_f$ does not have any associative subalgebra that contains all powers of $t$ (and then $f$ cannot lie in $S_0[t]$).
\end{lemma}

\begin{proof}
There exists a subset $X$ of $S_f$ which is a multiplicative group and contains all powers of $t$, if
and only if $ft\in Rf$ and $f\not\in Rt $ \cite[(8)]{P66}, i.e. if
and only if $t\in {\rm Nuc}_r(S_f)$ and $f\not\in Rt $. Now suppose $A$ is an associative subalgebra of $S_f$ that contains
all powers of $t$, choose $X=A$ and obtain that $t\in {\rm Nuc}_r(S_f)$ and $f\not\in Rt $.

If $f$ is irreducible, we know that $f\not\in Rt $.
If, additionally, $f\in S_0[t]$ then $S_0[t]/(f)$ is a subalgebra of $S_f$ that contains all powers of $t$, a contradiction.
\end{proof}

 \begin{proposition}
 For all $f\in S_0[t]$,
$S_0[t]/(f)$ is a commutative subring of $S_f$ and
$$S_0[t]/(f)=S_0\oplus S_0t\oplus\dots\oplus S_0t^{m-1}\subset{\rm Nuc}_r(S_f).$$
If ${\rm Nuc}_r(S_f)$ is larger than $S_0[t]/(f)$, then ${\rm Nuc}_r(S_f)$ is not commutative.
\end{proposition}

\begin{proof}
$S_f$ contains the commutative subring $S_0[t]/(f)$, where  $S_0={\rm Const}(\delta)\cap C(D)\cap {\rm Fix}(\sigma)$. This subring is isomorphic to the ring consisting of the
elements $\sum_{i=0}^{m-1}a_it^i$ with $a_i\in S_0$.
 In particular, we know that  the powers of $t$ are associative.
By Theorem  \cite[Theorem 4]{P15},
this implies that $t\in {\rm Nuc}_r(S_f)$. Clearly $S_0\subset {\rm Nuc}_r(S_f)$, so if $t\in {\rm Nuc}_r(S_f)$  then
 $S_0\oplus S_0t\oplus\dots\oplus S_0t^{m-1}\subset  {\rm Nuc}_r(S_f)$, hence we obtain the assertion.
 The last part is trivial then.
\end{proof}

 If $f\in S_0[t]$ is irreducible in $S_0[t]$, then $S_0[t]/(f)$ is an algebraic field extension of $S_0$ of degree $m$ contained in
 ${\rm Nuc}_r(S_f).$ Thus if $K$ is a finite field, $\delta=0$ and $f$ irreducible, then
${\rm Nuc}_r(S_f)=F\oplus Ft\oplus\dots\oplus Ft^{m-1}=F[t]/(f)$, employing the fact that in this case we know that the right nucleus has exactly $|F[t]/(f)|$ elements \cite{LS}.

\subsection{Right semi-invariant polynomials} We first investigate for which $f$ the algebra $D$ is
contained in the right nucleus of $S_f$. By \cite[Theorem 4]{P15}, this implies  that either
$S_f$ is associative or $\mathrm{Nuc}(S_f) = D$.

Recall that $f \in R=D[t;\sigma,\delta]$ is called \emph{right semi-invariant} if for every $a\in D$ there is $b\in D$ such that $f(t)a=bf(t)$ which is
equivalent to $fD \subseteq Df$. Similarly, $f$ is
\emph{left semi-invariant} if $Df \subseteq fD$ \cite{LL0, LLLM}. Moreover, $f$ is right semi-invariant if and only if
$df$ is right semi-invariant for all $d \in D^{\times}$ \cite[p.~8]{LL0}. Hence we  only need to consider monic $f$.
Furthermore, if $\sigma$ is an automorphism, then $f$ is right semi-invariant if and only if it is left semi-invariant if and only
if $fD = Df$ \cite[Proposition 2.7]{LL0}.
 Right semi-invariant polynomials canonically  arise in the theory of semi-linear transformations \cite{J37}.
For a thorough background on right semi-invariant polynomials see \cite{LL0, LLLM}.

If $f$ is semi-invariant and also satisfies $f(t)t=(bt+a)f(t)$ for some $a,b\in D$ then
 $f$ is called \emph{right invariant} which is equivalent to $fR\subset Rf$.
  If $f$ is right invariant then $Rf$ is a two-sided ideal in $R$ and conversely, every two-sided ideal in $R$ is generated by a
  right-invariant polynomial.
  That means $R$ is not simple
  if and only if there is a non-constant right-invariant $f\in R$. Moreover, assuming $\sigma$ is an automorphism,
  $R$ is not simple
  if and only if there is a non-constant monic semi-invariant $f\in R$ if and only if $\delta$ is a quasi-algebraic
  derivation  \cite{LLLM} (this last observation actually holds for any simple ring $D$).

\begin{theorem} \label{thm:semi-invariant iff D contained in E(f)}
$f \in R$ is right semi-invariant if and only if $D \subseteq \mathrm{Nuc}_r(S_f)$. In particular, if $f$ is right
semi-invariant, then either $\mathrm{Nuc}(S_f) = D$ or $S_f$ is associative.
\end{theorem}

\begin{proof}
If $f \in R$ is right semi-invariant, $fD \subseteq Df\subseteq Rf$ and hence $D \subseteq E(f) = \mathrm{Nuc}_r(S_f)$.
Conversely, if $D \subseteq \mathrm{Nuc}_r(S_f) = E(f)$ then for all $d \in D$, there exists $q(t) \in R$ such that
 $f(t)d = q(t)f(t)$.
Comparing degrees, we see $q(t) \in D$ and thus $fD \subseteq Df$.

The second assertion follows by \cite[Theorem 4]{P15}.
\end{proof}

\begin{proposition} \label{prop:lemonnier semi-invariant}
(\cite[(9.21)]{Le}). Suppose $\sigma$ is an automorphism of $D$, then the following are equivalent:
\\ (i) There exists a non-constant right semi-invariant polynomial in $R$.
\\ (ii) $R$ is not simple.
\\ (iii) There exist $b_0, \ldots, b_n \in D$ with $b_n \neq 0$ such that
$b_0 \delta_{c, \theta} + \sum_{i=1}^{n} b_i \delta^i = 0$, where
$\theta$ is an endomorphism of $D$ and $\delta_{c,\theta}$ denotes
the $\theta$-derivation of $D$ sending $x \in D$ to $c x - \theta(x)
c$.
\end{proposition}

Combining Theorem \ref{thm:semi-invariant iff D contained in E(f)}
and Proposition \ref{prop:lemonnier semi-invariant} we conclude:

\begin{corollary} \label{cor:simple and semi-invariant implies invariant}
Suppose $\sigma$ is an automorphism of $D$ and $R$ is simple. Then
there are no nonassociative algebras $S_f$ with $D \subseteq
\mathrm{Nuc}_r(S_f)$. In particular, there are no nonassociative
algebras $S_f$ with $D \subseteq \mathrm{Nuc}(S_f)$.
\end{corollary}

\begin{proof}
$R$ is not simple if and only if there exists a non-constant right
semi-invariant polynomial in $R$ by Proposition \ref{prop:lemonnier
semi-invariant}, and hence the assertion follows by Theorem
\ref{thm:semi-invariant iff D contained in E(f)}.
\end{proof}

Corollary \ref{cor:simple and semi-invariant implies invariant} actually also holds when $f\in S[t;\sigma,\delta]$,
where $S$ is only a simple ring and $\sigma$  an automorphism of $S$  \cite[Theorem 5.2]{LLLM}.

Recall that if $S$ is a division ring, or if $S$ is a simple ring and $\sigma\in{\rm
Aut}(S)$, then $R=S[t;\sigma,\delta]$ is not simple if and only if $\delta$ is
quasi-algebraic  \cite{LLLM}.

Recall also that $\sigma$ is an automorphism of $D$ \emph{of finite inner order} $k$ if
$\sigma^k = I_u$ for some $u \in D^{\times}$.
Using Theorem \ref{thm:semi-invariant iff D contained in E(f)} we can
rephrase the results \cite[Lemma 2.2, Corollary 2.12, Propositions 2.3 and 2.4]{LL0},
\cite[Corollary 2.6]{LLLM} on right semi-invariant polynomials in terms of the right nucleus of the nonassociative algebra $S_f$:

\begin{theorem} \label{thm:right semi invariant conditions}
 Let $f(t) = \sum_{i=0}^{m} a_i t^i \in
R$ be monic of degree $m$.
\\ (i) $D \subseteq \mathrm{Nuc}_r(S_f)$ if and only if $f(t)c = \sigma^m(c) f(t)$ for all $c \in D$, if and only if
\begin{equation} \label{eqn:right semi-invariant 1}
\sigma^m(c)a_j = \sum_{i=j}^{m} a_i \Delta_{i,j}(c)
\end{equation}
for all $c \in D$ and $j \in \{ 0, \ldots, m-1 \}$.
\\ (ii) Suppose $\sigma$ is an automorphism of $D$ of infinite inner order. Then $D \subseteq \mathrm{Nuc}_r(S_f)$
implies $S_f$ is associative.
\\ (iii) Suppose $\delta = 0$. Then $D \subseteq \mathrm{Nuc}_r(S_f)$ if and only if
\begin{equation} \label{eqn:right semi-invariant 2}
\sigma^m(c) = a_j \sigma^j(c) a_j^{-1}
\end{equation}
for all $c \in D$ and all $j \in \{ 0, \ldots, m-1 \}$ with $a_j \neq
0$. Furthermore, $S_f$ is associative if and only if $f(t)$ satisfies
\eqref{eqn:right semi-invariant 2} and $f(t)\in \mathrm{Fix}(\sigma)[t]\subset \mathrm{Fix}(\sigma)[t;\sigma]$.
\\ (iv) Suppose $\sigma = id$. Then $D \subseteq \mathrm{Nuc}_r(S_f)$ is equivalent to
\begin{equation} \label{eqn:right semi-invariant 3}
c a_j = \sum_{i=j}^{m} \binom{i}{j} a_i \delta^{i-j}(c),
\end{equation}
for all $c \in D$, $j \in  \{ 0, \ldots, m-1 \}$. Furthermore, $S_f$
is associative if and only if $f(t)$ satisfies \eqref{eqn:right
semi-invariant 3} and $f(t)\in \mathrm{Const}(\delta)[t]\subset \mathrm{Const}(\delta)[t;\delta] $.
\\ (v)  Suppose $\delta = 0$ and $\sigma$ is an automorphism of $D$ of finite inner order $k$, i.e. $\sigma^k=I_u$ for some $u\in D^\times$. Then the polynomials $g
\in D[t;\sigma]$ such that $D \subseteq \mathrm{Nuc}_r(S_g)$ are
precisely those of the form
\begin{equation} \label{eqn:right semi-invariant 4}
g(t)=b \sum_{j=0}^{n} c_j u^{n-j}t^{jk},
\end{equation}
where $n \in \mathbb{N}$, $c_n = 1$, $c_j \in F$ and $b \in
D^{\times}$. Furthermore, $S_g$ is associative if and only if $g(t)$
has the form \eqref{eqn:right semi-invariant 4} and
$g(t)\in \mathrm{Fix}(\sigma)[t]\subset \mathrm{Fix}(\sigma)[t;\sigma]$.
\end{theorem}

\subsection{Right $B$-weak semi-invariant polynomials}
Let now $B$ be a subring of $D$. We can find conditions on $f$ such that $B$
is contained in $ \mathrm{Nuc}_r(S_f)$ by generalizing the definition
of right semi-invariant polynomials as follows: we say $f \in D[t;\sigma,\delta]$ is \emph{(right) $B$-weak semi-invariant} if
$fB \subseteq D f$. Clearly any right semi-invariant polynomial is also $B$-weak semi-invariant for every  subring $B$ of $D$.

We call $f\in R$ a  \emph{(right) $B$-weak invariant}  polynomial if $f$ is
right $B$-weak semi-invariant and $f(t)t=(bt+a)f(t)$ for some $a,b\in B$.

 Note that when we have an extension of rings  $B\subset D$, which induces an extension of skew-polynomial rings
$B[t,\sigma,\delta]\subset D[t,\sigma,\delta]$ (i.e., $\sigma|_B=\sigma$, $\delta|_B=\delta$),
 every right semi-invariant $f\in B[t,\sigma,\delta]$ is right $B$-weak semi-invariant in $D[t,\sigma,\delta]$,
 and  every invariant $f\in B[t,\sigma,\delta]$ is right $B$-weak invariant in $D[t,\sigma,\delta]$.

\begin{example}
Let $K$ be a field, $\sigma$ be a non-trivial automorphism of $K$, $L
= \mathrm{Fix}(\sigma^j)$ be the fixed field of $\sigma^j$ for some $j >1$ and $f(t) = \sum_{i=0}^{n} a_i t^{ij} \in K[t;\sigma]$. Then
\begin{align*}
f(t)l &= \sum_{i=0}^{n} a_i t^{ij} l = \sum_{i=0}^{n} a_i\sigma^{ij}(l) t^{ij} = \sum_{i=0}^{n} a_i l t^{ij} = l f(t),
\end{align*}
for all $l \in L$ and hence $f L \subseteq L f$. In particular, $f$ is $L$-weak semi-invariant.
\end{example}

\begin{proposition} \label{prop:L-weak semi-invariant iff L subset Nuc_r(S_f)}
Let $B$ be a subring of $D$.
\\ (i) $f$ is $B$-weak semi-invariant if and only if $B \subseteq  \mathrm{Nuc}_r(S_f)$.
\\ (ii) If $f$ is $B$-weak semi-invariant but not right invariant, then $B \subseteq \mathrm{Nuc}(S_f) \subseteq D$.
\end{proposition}

\begin{proof}
(i) If $f \in R$ is $B$-weak semi-invariant, $fB \subseteq
Df \subseteq Rf$ and hence $B\subseteq  \mathrm{Nuc}_r(S_f)$. Conversely, if
$B \subseteq  \mathrm{Nuc}_r(S_f)$ then for all $b \in B$, there exists $q(t) \in R$
such that $f(t)b = q(t)f(t)$. Comparing degrees, we see $q(t) \in D$ and thus $f B \subseteq Df$.
\\ (ii) If $f$ is $B$-weak semi-invariant but not right invariant, the assertion follows from (i) and \cite[Theorem 4]{P15}.
\end{proof}

\begin{proposition} \label{prop:new(2)}
 Let $B$ be a subring of $D$ and $f\in R$ be a right $B$-weak invariant polynomial.
 Then $$B\oplus Bt\oplus\dots\oplus Bt^{m-1}\subset {\rm Nuc}_r(S_f).$$
\end{proposition}

\begin{proof}
(i) If $f\in R$ is a right $B$-weak invariant polynomial then $B\subset {\rm Nuc}_r(S_f)$
by Proposition \ref{prop:L-weak semi-invariant iff L subset Nuc_r(S_f)}. Since $f(t)t=(bt+a)f(t)$ for
some $a,b\in B$, we have $ft\in Rf$ which implies  $t\in {\rm Nuc}_r(S_f)$, hence the powers of $t$ are associative.
This in turn implies $t^mt=tt^m$ (\cite[Theorem 5]{P15} and \cite{P66}). Now ${\rm Nuc}_r(S_f)$ is an associative
subalgebra of $S_f$, thus $B\oplus Bt\oplus\dots\oplus Bt^{m-1}\subset {\rm Nuc}_r(S_f)$.
\end{proof}

We then obtain results similar to Theorem \ref{thm:right semi invariant conditions} (i), (iii) and (v) for $B$-weak
semi-invariant polynomials:

\begin{proposition} \label{prop:L-weak semi invariant conditions}
Let $f(t) = \sum_{i=0}^{m} a_i t^i \in D[t;\sigma,\delta]$ be monic of degree $m$ and $B$  a subring of $D$.
\\ (i) $f$ is $B$-weak semi-invariant if and only if $f(t)c = \sigma^m(c)f(t)$ for all $c \in B$, if and only if
\begin{equation} \label{eqn:L-weak semi invariant conditions 1}
\sigma^m(c) a_j = \sum_{i=j}^{m} a_i \Delta_{i,j}(c)
\end{equation}
for all $c \in B$, $j \in \{ 0, \ldots, m-1 \}$.
\\ (ii) Suppose $\delta = 0$. Then $f$ is $B$-weak semi-invariant if and only if  $\sigma^m(c) a_j = a_j \sigma^j(c)$
for all $c \in B$, $j \in \{ 0, \ldots, m-1 \}$.
\\(iii) Suppose $\sigma = id$. Then $f$ is $B$-weak semi-invariant if and only if
\begin{equation} \label{eqn:L-weak semi invariant conditions 2}
c a_j = \sum_{i=j}^{m} \binom{i}{j} a_i \delta^{i-j}(c)
\end{equation}
for all $c \in B$, $j \in \{ 0, \ldots, m-1 \}$.
\end{proposition}

\begin{proof}
(i) We have
\begin{equation} \label{eqn:L-weak semi invariant conditions 3}
f(t)c = \sum_{i=0}^{m} a_i t^i c = \sum_{i=0}^{m} a_i \sum_{j=0}^{i}
\Delta_{i,j}(c) t^j = \sum_{j=0}^{m} \sum_{i=j}^{m} a_i
\Delta_{i,j}(c) t^j
\end{equation}
for all $c \in B$, hence the $t^m$ coefficient of $f(t)c$ is
$\Delta_{m,m}(c) = \sigma^m(c)$, and so $f$ is $B$-weak semi-invariant if and only if $f(t)c = \sigma^m(c)f(t)$ for all $c
\in B$. Comparing the $t^j$ coefficient of \eqref{eqn:L-weak semi
invariant conditions 3} and $\sigma^m(c)f(t)$ for all $j \in \{ 0, \ldots, m-1 \}$ yields \eqref{eqn:L-weak semi invariant conditions
1}.
\\ (ii) When $\delta = 0$, $\Delta_{i,j} = 0$ unless $i = j$ in which case $\Delta_{j,j} = \sigma^j$.
Therefore \eqref{eqn:L-weak semi invariant conditions 1} simplifies
to $\sigma^m(c) a_j = a_j \sigma^j(c)$ for all $c \in B$, $j \in \{0, \ldots, m-1 \}$.
\\ (iii) When $\sigma = id$ we have
$$t^i c = \sum_{j=0}^{i} \binom{i}{j} \delta^{i-j}(c)$$
for all $c\in D$ by \cite[(1.1.26)]{J96} and thus
\begin{equation} \label{eqn:L-weak semi invariant conditions 4}
f(t) c = \sum_{i=0}^{m} a_i t^i c = \sum_{i=0}^{m} a_i \sum_{j=0}^{i}
\binom{i}{j} \delta^{i-j}(c) t^j = \sum_{j=0}^{m} \sum_{i=j}^{m}\binom{i}{j} a_i \delta^{i-j}(c) t^j
\end{equation}
for all $c \in B$. Furthermore $f$ is $B$-weak semi-invariant is equivalent to $f(t) c = c f(t)$ for all $c \in B$ by (i). Comparing
the $t^j$ coefficient of \eqref{eqn:L-weak semi invariant conditions 4} and $c f(t) = \sum_{i=0}^{m} c a_i t^i$ for all
$c \in B$, $j \in \{ 0, \ldots, m-1 \}$ yields \eqref{eqn:L-weak semi invariant conditions 2}.
\end{proof}

%
%

\section{(Right) division algebras obtained from Petit algebras} \label{section:When is S_f a Division Algebra?}

Petit algebras can be used to find classes of algebras that are right but not left division algebras.

 Let $D$ be a division algebra with center $F$ and $R=D[t;\sigma,\delta]$.
We say $f \in R$ is \emph{bounded} if there exists $0 \neq f^* \in
R$ such that $Rf^* = f^* R$ is the largest two-sided ideal of $R$
contained in $Rf$. The element $f^*$ is determined by $f$ up to
multiplication on the left by elements of $D^{\times}$.
The following result is stated but not proved in \cite[p.~13-07]{P66} and well known. We give a proof here for lack of a proper reference:

\begin{proposition} \label{prop:f irreducible implies E(f) division}
If $f \in R$ is irreducible then $E(f)$ is a division ring.
\end{proposition}

\begin{proof}
Let $\mathrm{End}_R(R/Rf)$ denote the endomorphism ring of the left
$R$-module $R/Rf$, that is $\mathrm{End}_R(R/Rf)$ consists of all
maps $\phi: R/Rf \rightarrow R/Rf$ such that $\phi(rh + r'h') = r
\phi(h) + r' \phi(h')$ for all $r,r' \in R$, $h, h' \in R/Rf$.

Now $f$ irreducible implies $R/Rf$ is a simple left $R$-module
\cite[p.~15]{G}, therefore $\mathrm{End}_R(R/Rf)$ is an associative
division ring by Schur's Lemma. Finally $E(f)$ is
isomorphic to the ring $\mathrm{End}_R(R/Rf)$ \cite[p.~18-19]{G}.
\end{proof}

\begin{remark}
If $\sigma$ is an automorphism and $f$ is bounded, then $f$ is
irreducible if and only if $E(f)={\rm Nuc}_r(S_f)$  is an associative
division algebra \cite[Proposition 4]{G}  which sums up classical
results from \cite{J43}. The condition that $f$ is bounded is necessary here, as is shown in \cite[Example 3]{G} where
$f\in\mathbb{Q}(x)[t;d/dt]$ is reducible in the differential operator ring $\mathbb{Q}(x)[t;d/dt]$, but ${\rm Nuc}_r(S_f)$ is a division algebra. For instance, if $D$ is a finite field and
$\delta = 0$, all polynomials are bounded and hence $f$ is
irreducible if and only if $E(f)$ is a finite field \cite[Theorem
3.3]{G0}.
\end{remark}

The argument leading up to \cite[Section 2., (6)]{P66} implies that
$S_f$ has no zero divisors if and only if $f$ is irreducible, which is in turn equivalent to
$S_f$ being a right division ring (i.e., right multiplication $R_a$ in $S_f$ is bijective for all $0\not=a\in S_f$).
We include a proof for the convenience of the reader, since no full
proof is given in \cite{P66}:

\begin{theorem} \label{thm:f(t) irreducible iff S_f right division}  \label{lem:L_a injective but not nec surjective}
(\cite[(6)]{P66}). Let $f \in R$ have degree $m$ and $0 \neq a \in S_f$.
\\ (i) $R_a$ is bijective is equivalent to $1$ being a
greatest common right divisor of $f$ and $a$ (i.e., $Da(t)+Df(t)=D$).
\\ (ii) $S_f$ is a right division
algebra if and only if $f$ is irreducible, if and only if $S_f$ has no zero divisors.
\\ (iii) If $f$ is irreducible then $L_a$ is injective for all $0\not=a\in S_f$.
\end{theorem}

\begin{proof}
(i) Let $0 \neq a \in S_f$. Since $S_f$ is a free left $D$-module of
finite rank $m$ and $R_a$ is left $D$-linear, $R_a$ is bijective if
and only if it is injective \cite[Chapter IV, Corollary
2.14]{hungerford1980algebra}, which is equivalent to
$\mathrm{ker}(R_a) = \{ 0 \}$. Now $R_a(z) = z  a = 0$ is
equivalent to $za \in Rf$, which means we can write
$$\mathrm{ker}(R_a) = \{ z \in R_m \ \vert \ za \in Rf \}.$$
Furthermore, $R$ is a left principal ideal domain, which implies $za
\in Rf$ if and only if $za \in Ra \cap Rf = Rg = Rha,$ where $g = ha$
is the least common left multiple of $a$ and $f$. Therefore $za \in
Rf$ is equivalent to $z \in Rh$, and hence $\mathrm{ker}(R_a) \neq \{
0 \}$, if and only if there exists a polynomial of degree strictly
less than $m$ in $Rh$, which is equivalent to $\mathrm{deg}(h) \leq
m-1$.

Let $b \in R$ be a right greatest common divisor of $a$ and $f$. Then
$\mathrm{deg}(f) + \mathrm{deg}(a) = \mathrm{deg}(g) +\mathrm{deg}(b) = \mathrm{deg}(ha) + \mathrm{deg}(b)$ by
\cite[Proposition 1.3.1]{J96}, and so $\mathrm{deg}(b) =\mathrm{deg}(f) - \mathrm{deg}(h).$ Thus $\mathrm{deg}(h) \leq m-1$
if and only if $\mathrm{deg}(b) \geq 1$. We conclude
$\mathrm{ker}(R_a) = \{ 0 \}$ if and only if $\mathrm{deg}(b) = 0$, if and only if $1$ is a right greatest common divisor of $f(t)$ and
$a$. In particular, $S_f$ is a right division algebra if
and only if $R_a$ is bijective for all $0 \neq a \in S_f$, if and
only if $1$ is a right greatest common divisor of $f(t)$ and $a$ for all $0 \neq a \in S_f$, if and only if $f(t)$ is
irreducible.
\\ (ii) If $f$ is irreducible then $L_a$ and $R_a$ are
injective for all $0\not=a\in S_f$ (i), therefore $S_f$ has no zero divisors.
The converse of the last equivalence statement is trivial.
\\ (iii) If $R_h$ is bijective this automatically implies that $L_h$ is injective, for all $0\not=h\in S_f$.
\end{proof}

\begin{lemma} \label{prop:S_f associative division iff irreducible}
 If $f \in R$ is right invariant, then  $S_f$ is associative and a division algebra if and only if $f$ is irreducible.
\end{lemma}

\begin{proof}
Suppose $f$ is right invariant. Then $S_f$ is associative by \cite[Theorem 4]{P15}. If $f$ is reducible then $S_f$ is
trivially not a division algebra. Conversely, if $f$ is irreducible the maps $R_b$ are bijective for all $0 \neq b \in S_f$
by Theorem \ref{thm:f(t) irreducible iff S_f right division}. This implies the maps $L_b$ are also bijective for all $0 \neq b \in S_f$
by \cite[Lemma 1B]{B46}, and so $S_f$ is a division algebra.
\end{proof}

This implies a generalization of Theorem \cite[Theorem 4]{P15}:

\begin{theorem} \label{thm:L_t surjective iff sigma surjective}
Let $f(t) =  \sum_{i=0}^{m} a_i t^i \in D[t;\sigma]$ be monic and
$a_0 \neq 0$. Then for every $j \in \{ 1, \ldots, m-1 \}$, $L_{t^j}$
is surjective if and only if $\sigma$ is surjective. In particular,
if $\sigma$ is not surjective then $S_f$ is not a left division
algebra.
\end{theorem}

\begin{proof}
We first prove the result for $j = 1$: Given $z = \sum_{i=0}^{m-1}
z_i t^i \in S_f$, we have
\begin{equation} \label{eqn:L_t surjective iff sigma surjective 1}
\begin{split}
L_t(z) &= t \circ z = \sum_{i=0}^{m-2} \sigma(z_{i})t^{i+1} +
\sigma(z_{m-1})t \circ t^{m-1} \\ &= \sum_{i=1}^{m-1}
\sigma(z_{i-1})t^i + \sigma(z_{m-1}) \sum_{i=0}^{m-1} a_i t^i.
\end{split}
\end{equation}
$\Rightarrow$: Suppose $L_t$ is surjective, then given any $b \in D$ there exists $z \in S_f$ such that $t \circ z = b$.
The $t^0$-coefficient of $L_t(z)$ is $\sigma(z_{m-1}) a_0$ by \eqref{eqn:L_t surjective iff sigma surjective 1}, and thus
for all $b \in D$ there exists $z_{m-1} \in D$ such that $\sigma(z_{m-1}) a_0 = b$. Therefore $\sigma$ is surjective.
\\ $\Leftarrow$: Suppose $\sigma$ is surjective and let $g = \sum_{i=0}^{m-1} g_i t^i \in S_f$. Define
$$z_{m-1} = \sigma^{-1}(g_0 a_0^{-1} ), \ z_{i-1} = \sigma^{-1}(g_i) - z_{m-1} \sigma^{-1}(a_i)$$ for all
 $i \in \{ 1, \ldots , m-1 \}$. Then
\begin{equation*}
\begin{split}
L_t(z) &= \sigma(z_{m-1}) a_0 + \sum_{i=1}^{m-1} \big(
\sigma(z_{i-1}) + \sigma(z_{m-1}) a_i \big) t^i = \sum_{i=0}^{m-1} g_i
t^i = g,
\end{split}
\end{equation*}
by \eqref{eqn:L_t surjective iff sigma surjective 1}, which implies
$L_t$ is surjective.

Hence $L_t$ surjective is equivalent to $\sigma$ surjective. To prove
the result for all $j \in \{ 1, \ldots, m-1 \}$ we show that
\begin{equation} \label{eqn:L_t surjective iff sigma surjective 2}
L_{t^j} = L_t^j,
\end{equation}
for all $j \in \{ 1, \ldots, m-1 \}$, then it follows $\sigma$ is
surjective if and only if $L_t$ is surjective if and only if $L_t^j =
L_{t^j}$ is surjective. In the special case when $D = \mathbb{F}_q$
is a finite field, $\sigma$ is an automorphism and $f$ is monic
and irreducible, the equality \eqref{eqn:L_t surjective iff sigma
surjective 2} is proven in \cite[p.~12]{LS}. A similar proof also
works in our context: suppose inductively that
$L_{t^j} = L_t^j$ for some $j \in \{ 1, \ldots, m-2 \}$. Then
$L_t^j(b) = t^j b \ \mathrm{mod}_r f$ for all $b \in R_m$. Let
$L_t^j(b) = b'$ so that $t^j b = qf + b'$ for some $q \in R$. We have
\begin{align*}
L_t^{j+1}(b) &= L_t(L_t^j(b)) = L_t(b') = L_t(t^j b - q f) = t \circ
(t^j b - q f) \\ &= (t^{j+1} b - tqf) \ \mathrm{mod}_r f = t^{j+1}b \
\mathrm{mod}_r f = L_{t^{j+1}}(b),
\end{align*}
hence \eqref{eqn:L_t surjective iff sigma surjective 2} follows by
induction.
\end{proof}

Recall that for $\delta=0$, $L_t$ is a \emph{pseudo-linear transformation}, i.e. $L_t(ah(t))=\sigma(a)L_t(h(t))$ for all $a\in
S$, $h(t)\in S_f$,  and that $L_h=h(t)(L_t)=\sum_{i=0}^{m-1}a_iL_t^i$
for $h(t)=\sum_{i=0}^{m-1}a_it^i$.
If $f$ is irreducible, then $L_t$ is irreducible, that means the only $L_t$-invariant subspaces
of the left $D$-module $D^m$ are $\{0\}$ and $D^m$, as pointed out in  \cite{LS}.

\begin{corollary} \label{cor:S_f right but not left division algebra}
Suppose $\sigma$ is not surjective and $f \in D[t;\sigma]$ is
irreducible. Then $S_f$  has no zero divisors and is a right division algebra but not a left
division algebra. In particular, $S_f$ is an infinite-dimensional  $S_0$-algebra.
\end{corollary}

The following result was stated but not proved by Petit
\cite[(7)]{P66}:
\begin{theorem} \label{thm:S_f_division_iff_irreducible}
 Let $f \in D[t; \sigma, \delta]$ be such that
$S_f$ is either a finite-dimensional $S_0$-vector space or a right
$\mathrm{Nuc}_r(S_f)$-module which is free of finite rank. Then $S_f$
is a division algebra if and only if $f$ is irreducible.
\end{theorem}

\begin{proof}
When $S_f$ is associative the assertion follows by Lemma
\ref{prop:S_f associative division iff irreducible} so suppose $S_f$
is not associative. If $f$ is reducible, $S_f$ is not a division
algebra. Conversely, suppose $f$ is irreducible, so that $S_f$ is
a right division algebra by Theorem \ref{thm:f(t) irreducible iff S_f
right division}. Let $0 \neq a \in S_f$ be arbitrary, then $L_a$ is
injective for all $0 \neq a \in S_f$ by Lemma \ref{lem:L_a injective
but not nec surjective}. We prove $L_a$ is surjective, hence that
$S_f$ is also a left division algebra:
\\ (i) If $S_f$ is a finite-dimensional $S_0$-vector space then
 $L_a$ is clearly surjective by \cite[Chapter IV, Corollary 2.14]{hungerford1980algebra}, since  $L_a$ is $F$-linear.
\\ (ii) Suppose $S_f$ is a free right $\mathrm{Nuc}_r(S_f)$-module of finite rank, then $E(f)$ is a division ring
by Proposition \ref{prop:f irreducible implies E(f) division}.
Furthermore,
$L_a$ is right $\mathrm{Nuc}_r(S_f)$-linear. Therefore $L_a$ is again
surjective by \cite[Chapter IV, Corollary 2.14]{hungerford1980algebra}.
\end{proof}

\begin{theorem} \label{thm:L weak irreducible iff division}
Let $\sigma$ be an automorphism of $D$, $B$ be a  subring of $D$ such
that $D$ is a free right $B$-module of finite rank, and $f \in D[t;\sigma,\delta]$ be $B$-weak semi-invariant. Then
$S_f$ is a
division algebra if and only if $f$ is irreducible. In particular, if $\sigma$ is an automorphism of $D$ and $f$ is right
semi-invariant then $S_f$ is a division algebra if and only if $f$ is irreducible.
\end{theorem}

\begin{proof}
If $f$ is reducible then $S_f$ is not a division algebra.
Conversely, suppose $f$ is irreducible. Then $S_f$ is a right
division algebra by Theorem \ref{thm:f(t) irreducible iff S_f right
division} so we are left to show $S_f$ is also a left division
algebra. Let $0 \neq a \in S_f$ be arbitrary and recall $L_a$ is
injective by Lemma \ref{lem:L_a injective but not nec surjective}.
Since $f$ is $B$-weak semi-invariant, $B \subseteq \mathrm{Nuc}_r(S_f)$ which implies that
 $L_a$ is right $B$-linear. $S_f$ is a free right $D$-module of rank $m = \mathrm{deg}(f)$
because $\sigma$ is an automorphism. Since $D$ is a free right
$B$-module of finite rank then also $S_f$ is a free right $B$-module
of finite rank. Thus \cite[Chapter IV, Corollary
2.14]{hungerford1980algebra} implies $L_a$ is bijective as required.
\end{proof}

\cite[Theorem 4]{C} yields:

\begin{theorem} \label{thm:bounded}
Let  $f\in R= D[t;\sigma,\delta]$ be irreducible. Then $f$ is bounded
if and only if $S_f$ is free of finite rank as a $ {\rm
Nuc}_r(S_f)$-module. In this case, $S_f$ is a division algebra.
\end{theorem}

\begin{proof}
The first part of the statement is \cite[Theorem 4]{C}. Since $f$
irreducible, $S_f$ is a right division algebra and $L_a$ is injective
for all  $0\not=a\in S_f$ as observed in \cite[Section 2.,
(7)]{P66}.  The second part then follows from the fact that $S_f$ is
free of finite rank as a $ {\rm Nuc}_r(S_f)$-module,
which means the injective  $ {\rm Nuc}_r(S_f)$-linear map $L_a$ is also surjective.
\end{proof}

For $\sigma=0$ this is \cite[Theorem 2]{P17}.

\begin{corollary}  \label{cor:bounded}
Let $f\in R= D[t;\sigma,\delta]$ be irreducible.
\\ (i) Let $\sigma$ be surjective and  $D={\rm Nuc}_r(S_f)$. Then $f$ is bounded and $S_f$ is a division algebra.
\\ (ii) Let  $f$ be bounded, then $S_f$ is a division algebra.
\end{corollary}

\begin{proof}
(i) If $\sigma$ is surjective then $S_f$ is a right $D$-module, free of rank $m$. Since $D={\rm Nuc}_r(S_f)$, Theorem \ref{thm:bounded}
yields the assertion.
\\ (ii) is trivial.
\end{proof}

If $\sigma$ is an automorphism, $R= D[t;\sigma,\delta]$ has finite rank over its center if and only if $D$
is of finite rank over
$C_t=\{a\in F\,|\, at=ta\}$ if and only if all polynomials of $R$ are bounded and if for all $f$ of degree non-zero, ${\rm
deg}(f^*)/{\rm deg}(f)$ is bounded in $\mathbb{Q}$ ($f^*$ being the bound of $f$) \cite[Theorem IV]{CI}. Since
$C_t= {\rm Const}(\delta)\cap {\rm Fix}(\sigma)=S_0\subset F$ we conclude:

\begin{proposition} \label{prop:important}
Assume $R= D[t;\sigma,\delta]$, $\sigma$ is an automorphism,  and one
of the two following equivalent conditions hold:
\\ (i) $R=D[t;\sigma,\delta]$ has finite rank over its center;
\\ (ii) $D$ has finite rank over $S_0$.
\\ Then every $f\in R$ is bounded. In particular, if $f$ is irreducible then $S_f$ is a division algebra.
\end{proposition}

 For $\sigma=0$, this is  \cite[Proposition 3]{P17}.

Suppose now that $\sigma$ is an automorphism. Then $S_f$ is a free right $D$-module of rank $m$
and since $L_a$ is ${\rm Nuc}_r(S_f)$-linear for any non-zero $a\in S_f$, in this case $S_f$ is a division algebra for an
irreducible $f$ if $D\subset {\rm Nuc}_r(S_f)$ or if there is a subalgebra
$B\subset D$ such that $B\subset {\rm Nuc}_r(S_f)$ and $D $ has finite rank as a right $B$-module
 (these conditions were not stated in \cite[p.~13-14]{P66} but seem necessary). We obtain:

\begin{proposition}  \label{prop:new(3)}
 Suppose that $\sigma$  is an automorphism and $f$ is irreducible.
\\ (i) If $D\subset {\rm Nuc}_r(S_f)$ then $S_f$ is a division algebra.
\\ (ii) If there is a subalgebra $B\subset D$ such that
$B\subset {\rm Nuc}_r(S_f)$ and  $D $ is free of finite rank as a right $B$-module
then $S_f$ is a division algebra.
\end{proposition}

\begin{proof}
 $S_f$ is a right $D$-module and left multiplication $L_a$ is ${\rm Nuc}_r(S_f)$-linear, so in particular $D$-linear.
Since $f$ is irreducible,  $L_a$ is injective for all nonzero
$a\in S_f$.
 If $D\subset {\rm Nuc}_r(S_f)$
then $S_f$ is a free right $D$-module of rank $m$, and if there is a
subalgebra $B\subset D$ such that $B\subset {\rm Nuc}_r(S_f)$ and $S$ is free of finite rank as a right $B$-module, then
 $S_f$ is a free right $B$-module of finite rank. Thus $L_a$ is bijective for all nonzero $a\in S_f$.
\end{proof}

%
%

\section{Irreducibility criteria for some polynomials  in $R = D[t;\sigma]$}
\label{section:Irreducibility Criteria in D[t;sigma]}

Let $D$ be a division algebra over $F$ and $f(t) = t^m - \sum_{i=0}^{m-1} a_i t^i \in R = D[t;\sigma]$.

\subsection{} There are already several irreducibility criteria for $f$ available
 in the literature. We start by collecting some that are useful for constructing (right) division algebras $S_f$
 for the convenience of the reader.

We first determine the
remainder after dividing $f(t)$ on the right by $t-b$ where $ b \in D$. By \cite[p.~15ff]{J96} we have
 $(t-b) \vert_r f(t)$ is
equivalent to
\begin{equation}  \label{prop:rightdivdegree1}
a_m N_m(b) - \sum_{i=0}^{m-1} a_i N_i (b) = 0
\end{equation}
 where $N_i (b) = \sigma^{i-1} (b) \cdots \sigma(b) b$ for $i
> 0$ and $N_0 (b) = 1$, i.e. to this remainder being zero.

When $\sigma$ is an automorphism of $D$, we can also determine the
remainder after dividing $f(t)$ on the left by $(t-b), \ b \in D$:
Similarly to \cite[p.~15ff]{J96}
we have
\begin{equation} \label{eqn:Left division identity}
\begin{split}
t^i - b \sigma^{-1}(b)& \cdots \sigma^{1-i}(b) = (t-b) \Big(t^{i-1 }
+ \sigma^{-1}(b)t^{i-2}+ \\ & \sigma^{-1}(b) \sigma^{-2}(b)t^{i-3} +
\ldots + \sigma^{-1}(b) \sigma^{-2}(b) \cdots \sigma^{1-i}(b)\Big)
\end{split}
\end{equation}
for all $i \in \mathbb{N}$. Multiplying \eqref{eqn:Left division
identity} on the right by $\sigma^{-i}(a_i)$, and using $a_i t^i =
t^i \sigma^{-i} (a_i)$ gives
\begin{equation*}
\begin{split}
&a_i t^i - b \sigma^{-1}(b) \cdots \sigma^{1-i}(b) \sigma^{-i}(a_i)
\\ &= (t-b) \Big(t^{i-1} + \sigma^{-1}(b)t^{i-2} + \ldots +
\sigma^{-1}(b) \sigma^{-2}(b) \cdots \sigma^{1-i}(b) \Big)
\sigma^{-i}(a_i).
\end{split}
\end{equation*}
Summing over $i$, we obtain $$f(t) = (t-b)q(t) + M_m (b) -
\sum_{i=0}^{m-1} M_i(b) \sigma^{-i}(a_i),$$ for some $q(t) \in R$
where  $M_0 (b) = 1$, $M_1(b) = b$ and $M_i(b) = b \sigma^{-1}(b)
\cdots \sigma^{1-i}(b)$ for $i \geq 2$. We immediately conclude:

\begin{proposition} \label{prop:leftdivdegree1}
Suppose $\sigma$ is an automorphism. Then $(t-b) \vert_l f(t)
\ $ if and only if $ \ M_m (b) - \sum_{i=0}^{m-1} M_i(b)
\sigma^{-i}(a_i) = 0.$
\end{proposition}

\begin{corollary} \label{cor:left iff right divisor}
Suppose $\sigma$ is an automorphism and $f(t) = t^m - a \in
D[t;\sigma]$. Then $f(t)$ has a left linear divisor if and only if it
has a right linear divisor.
\end{corollary}

\begin{proof}
Let $b \in D$, then $(t-b) \vert_r f(t)$ is equivalent to
$\sigma^{m-1}(b) \cdots \sigma(b) b = a$ by
\eqref{prop:rightdivdegree1}, if and only if $c \sigma^{-1}(c) \cdots
\sigma^{1-m}(c) = a$ where $c = \sigma^{m-1}(b)$, if and only if
$(t-c) \vert_l f(t)$ by Proposition \ref{prop:leftdivdegree1}.
\end{proof}

 The following
were stated but not proven by Petit in \cite[(17), (18)]{P66} and are
direct consequences of  \eqref{prop:rightdivdegree1} and Proposition
\ref{prop:leftdivdegree1}:

\begin{theorem} \label{thm:Petit_factor}
(i)  $f(t) = t^2 - a_1 t - a_0 \in D[t;\sigma]$ is irreducible if and only if
\begin{equation*} \label{eqn:Petit (17)}
\sigma(b)b - a_1 b - a_0 \neq 0,
\end{equation*}
for all $b \in D.$
\\ (ii) Suppose $\sigma$ is an automorphism. $f(t) = t^3 - a_2 t^2 - a_1 t - a_0 \in D[t;\sigma]$ is irreducible if and only if
\begin{equation*} \label{eqn:Petit (18) 1}
\sigma^2 (b) \sigma(b) b - \sigma^2 (b)\sigma(b) a_2 - \sigma^2 (b)
\sigma(a_1) - \sigma^2 (a_0) \neq 0,
\end{equation*}
and
\begin{equation*} \label{eqn:Petit (18) 2}
\sigma^2 (b) \sigma(b) b - a_2 \sigma(b) b - a_1 b - a_0 \neq 0,
\end{equation*}
for all $b \in D.$
\end{theorem}

\begin{proof}
(i) $f(t)$ is irreducible if and only if $t-b \nmid_r f(t)$ for all $b \in D$, if and only if
$N_2 (b) - a_1 N_1 (b) - a_0 N_0(b) = \sigma(b) b - a_1 b - a_0 \neq
0,$ for all $b \in D$ by  \eqref{prop:rightdivdegree1}.
\\ (ii)  $f(t)$ is irreducible if and only if $t-b \nmid_r f(t)$ and $t-b \nmid_l f(t)$ for all $b \in D$, if and only if
\begin{equation*} \label{eqn:Petit_factor 1}
\sigma^2 (b) \sigma(b) b - a_2 \sigma(b)b - a_1 b - a_0 \neq 0,
\end{equation*}
and
\begin{equation} \label{eqn:Petit_factor 2}
b \sigma^{-1}(b) \sigma^{-2}(b) - b \sigma^{-1}(b) \sigma^{-2}(a_2) -
b \sigma^{-1}(a_1) - a_0 \neq 0,
\end{equation}
for all $b \in D$ by  \eqref{prop:rightdivdegree1} and Proposition
\ref{prop:leftdivdegree1}. Applying $\sigma^2$ to
\eqref{eqn:Petit_factor 2} we obtain the assertion.
\end{proof}

\begin{corollary} \label{cor:Irreducibility t^3-a}
Suppose $\sigma$ is an automorphism, then $f(t) = t^3 - a \in
D[t;\sigma]$ is irreducible if and only if $\sigma^2 (b) \sigma(b) b
\neq a$ for all $b \in D$.
\end{corollary}

\begin{proof}
By Corollary \ref{cor:left iff right divisor}, $f(t)$ has a right
linear divisor if and only if it has a left linear divisor. Therefore
$f(t)$ is irreducible if and only if $(t-b) \nmid_r f(t)$ for all $b
\in D$, if and only if $\sigma^2 (b) \sigma(b) b \neq a$ for all $b
\in D$  by \eqref{prop:rightdivdegree1}.
\end{proof}

\begin{lemma} \label{lem:f(bt)=q(bt)g(bt)}
Let $f(t) \in R = D[t;\sigma]$ and suppose $f(t) = q(t) g(t)$ for
some $q(t), g(t) \in R$. Then $f(bt) = q(bt) g(bt)$ for all $b \in
S_0 = F \cap \mathrm{Fix}(\sigma)$.
\end{lemma}

\begin{proof}
Write $q(t) = \sum_{i=0}^{l} q_i t^i$, \ $g(t) = \sum_{j=0}^{n} g_j
t^j$, then $$f(t) = q(t) g(t) = \sum_{i=0}^{l} \sum_{j=0}^{n} q_i t^i
g_j t^j = \sum_{i=0}^{l} \sum_{j=0}^{n} q_i \sigma^i(g_j) t^{i+j},$$
and so
\begin{align*}
q(bt) g(bt) &= \sum_{i=0}^{l} q_i (bt)^i \sum_{j=0}^{n} g_j (bt)^j =
\sum_{i=0}^{l} \sum_{j=0}^{n} q_i \sigma^i(g_j) b^{i+j} t^{i+j} \\ &=
\sum_{i=0}^{l} \sum_{j=0}^{n} q_i \sigma^i(g_j) (bt)^{i+j} = f(bt),
\\
\end{align*}
for all $b \in S_0$.
\end{proof}

The following result was stated as an Exercise by Bourbaki in
\cite[p.~344]{bourbaki1973elements} and proven in the special case
where $\sigma$ is an automorphism of order $m$ in \cite[Proposition
3.7.5]{cohn1995skew}. We include a proof as we will use this method a
second time for the proof of Theorem \ref{thm:Petit(19), delta not 0}.

\begin{theorem} \label{thm:bourbaki}
Let $\sigma$ be an endomorphism of $D$, $f(t) = t^m - a \in R =
D[t;\sigma]$ and suppose $S_0 = F \cap \mathrm{Fix}(\sigma)$ contains
a primitive $m$th root of unity  $\omega $. If $g(t) \in R$
is a monic irreducible polynomial dividing $f(t)$ on the right, then
the degree $d$ of $g(t)$ divides $m$ and $f(t)$ is the product of
$m/d$ polynomials of degree $d$.
\end{theorem}

\begin{proof}
Let $g(t) \in R$ be a monic irreducible polynomial of degree $d$
dividing $f(t)$ on the right.

Define $g_{i} (t) = g(\omega^i t)$ for all $i \in \{ 0, \ldots,
m-1\}$. Then $\bigcap_{i=0}^{m-1} R g_{i}(t)$ is an ideal of $R$, and
since $R$ is a left principle ideal domain, we have
\begin{equation} \label{eqn:bourbaki 1}
Rh(t) = \bigcap_{i=0}^{m-1} R g_{i}(t),
\end{equation}
for a suitably chosen $h(t) \in R$. Furthermore, we may assume $h(t)$
is monic, otherwise if $h(t)$ has leading coefficient $d \in
D^{\times}$, then $Rh(t) = R(d^{-1}h(t))$.

We show $f(t) \in Rh(t)$: As $g(t)$ right divides $f(t)$, we can
write $f(t) = q(t)g(t)$ for some $q(t) \in R$. In addition, we have
$(\omega t)^i = \omega^i t^i$ for all $i \in \{ 0, \ldots, m-1 \}$
because $\omega \in S_0$, therefore $$f(\omega^i t) = \omega^{mi}
t^m-a = t^m - a = f(t) = q(\omega^i t) g(\omega^i t)$$ by Lemma
\ref{lem:f(bt)=q(bt)g(bt)} and so $g_{i}(t)$ right divides $f(t)$ for
all $i \in \{ 0, \ldots, m-1 \}$. This means $$f(t) \in
\bigcap_{i=0}^{m-1} R g_{i}(t) = Rh(t),$$ in particular, $Rh(t)$ is
not the zero ideal.

We next show $h(\omega^i t) = h(t)$ for all $i \in \{ 0, \ldots , m-1
\}$: we only do this for $i = 1$, the other cases are similar. Notice
$h(t) \in \bigcap_{j=0}^{m-1} R g_{j}(t)$ by \eqref{eqn:bourbaki 1}
and thus there exists $q_0(t), \ldots, q_{m-1}(t) \in R$ such that
$h(t) = q_j(t) g_{j}(t)$, for all  $j \in \{ 0, \ldots , m-1 \}$.
Therefore $$h(\omega t) = q_{m-1}(\omega t) g_{m-1}(\omega t) =
q_{m-1}(\omega t) g_0(t),$$ and $$h(\omega t) = q_j(\omega t)
g_{j}(\omega t) = q_j(\omega t) g_{j+1}(t) \in Rg_{j+1}(t),$$ for all
$j \in \{ 0, \ldots , m-2 \}$ by Lemma \ref{lem:f(bt)=q(bt)g(bt)},
which implies $$h(\omega t) \in \bigcap_{j=0}^{m-1} R g_{j}(t) =
Rh(t).$$ As a result $h(\omega t) = k (t) h(t)$ for some $k(t) \in
R$, and by comparing degrees, we conclude $0 \neq k(t) = k \in D$.
Suppose $h(t)$ has degree $l$ and write $$h(t) = a_0 + \ldots +
a_{l-1}t^{l-1} + t^l, \ a_j \in D,$$ here $Rg(t) \supseteq Rh(t)$ and
$f(t) \in Rh(t)$ which yields $\mathrm{deg}(g(t)) = d \leq l \leq m$.
Since $h(\omega t) = k h(t)$, we have $k t^l = (\omega t)^l = \big(
\prod_{j=0}^{l-1} \sigma^j(\omega) \big) t^l = \omega^l t^l$ which
implies $k = \omega^l$. Clearly, the coefficients $a_j$ must be zero
for all $j \in \{ 1, \ldots, l-1 \}$, otherwise $a_j (\omega t)^j = k
a_j t^j = \omega^l a_j t^j$ giving $\omega^j = \omega^l$, a
contradiction as $\omega$ is a primitive $m^{\text{th}}$ root of
unity. This means $h(t) = t^l + a_0$, and with
$\omega^l t^l + a_0 =h(\omega t) = k h(t) = \omega^l (t^l + a_0) = \omega^l t^l + \omega^l a_0,$
we obtain $\omega^l = 1$. This implies $l = m$ and $k =
\omega^m = 1$, hence $h(\omega t) = h(t)$.

We next prove $h(t) = f(t)$: Now $f(t) \in Rh(t)$ implies $f(t) = t^m
- a = p(t)(t^m + a_0)$ for some $p \in R$. Comparing degrees we see
$p \in D^{\times}$, thus $t^m - a = p(t^m + a_0) = pt^m + p a_0$
which yields $p=1$ , $a_0 = -a$ and $f(t) = h(t)$.

Finally, $\bigcap_{i=0}^{m-1} R g_{i}(t) = Rf(t)$ is equivalent to
$f(t)$ being the least common left multiple  of the $g_{i}(t)$, $i
\in \{ 0, \ldots, m-1 \}$ \cite[p.~10]{J96}. As a result, we can
write $$f(t) = q_{i_r}(t) q_{i_{r-1}}(t) \cdots q_{i_1}(t),$$ by
\cite[p.~496]{O1}, where $i_1 = 0 < i_2 < \ldots < i_r \leq m-1$ and
each $q_{i_s}(t) \in R$ is similar to $g_{i_s}(t)$. Similar
polynomials have the same degree \cite[p.~14]{J96} so $r = m/d$, and
$f(t)$ factorises into $m/d$ irreducible polynomials of degree $d$.
\end{proof}

Theorem \ref{thm:bourbaki} implies \cite[Lemma 10]{Am2}, cf. also \cite[Theorem 6 (iii)]{P15.2}, which
improves \cite[(19)]{P66}:

\begin{theorem} \label{thm:Petit(19)}
Suppose $m$ is prime, $\sigma$ is an endomorphism of $D$ and $S_0 = F \cap \mathrm{Fix}(\sigma)$
contains a primitive $m$th root of unity. Then $f(t) = t^m
- a \in D[t;\sigma]$ is irreducible if and only if it has no right
linear divisors, if and only if $$a \neq \sigma^{m-1} (b) \cdots
\sigma(b) b$$
 for all $b \in D$.
\end{theorem}

\begin{proof}
Let $g(t) \in D[t;\sigma]$ be an irreducible polynomial of degree $d$
dividing $f(t)$ on the right. Without loss of generality $g(t)$ is
monic, otherwise if $g(t)$ has leading coefficient $c \in
D^{\times}$, then $c^{-1}g(t)$ is monic and also right divides
$f(t)$. Thus $d$ divides $m$ by Theorem \ref{thm:bourbaki} and since
$m$ is prime, either $d = m$, in which case $g(t) = f(t)$, or $d =
1$, which means $f(t)$ can be written as a product of $m$ linear
factors. Therefore $f(t)$ is irreducible if and only if $(t-b)
\nmid_r f(t)$ for all $b \in D$, if and only if $a \neq \sigma^{m-1}
(b) \cdots \sigma(b) b$, for all $b \in D$ by
\eqref{prop:rightdivdegree1}.
\end{proof}

\subsection{Skew polynomials of degree four}
Suppose $\sigma$ is an automorphism of $D$ and $f(t) = t^4 - a_3 t^3
- a_2 t^2 - a_1 t - a_0 \in R = D[t;\sigma]$. Then either $f(t)$ is
irreducible, $f(t)$ is divisible by a linear factor from the right,
from the left, or $f(t) = g (t) h(t)$ for some $g(t), h(t) \in R$ of
degree $2$. In \eqref{prop:rightdivdegree1} and Proposition
\ref{prop:leftdivdegree1} we computed the remainders after dividing
$f(t)$ by a linear polynomial on the right and the left. We now
compute the remainder after dividing $f(t)$ by $t^2 - c t - d$ on the
right,  with $ c, d \in D$. To do this we use the identities
\begin{equation} \label{eqn:degree 4 t^2 identity 1}
t^2 = (t^2 - c t - d) + (c t + d),
\end{equation}
\begin{equation} \label{eqn:degree 4 t^2 identity 2}
t^3 = (t + \sigma(c)) \big( t^2 - c t - d \big) + \big( \sigma(d) +
\sigma(c)c \big) t + \sigma(c) d,
\end{equation}
and
\begin{equation} \label{eqn:degree 4 t^2 identity 3}
\begin{split}
t^4 &= \big( t^2 + \sigma^2 (c) t + \sigma^2 (d) + \sigma^2 (c)
\sigma(c) \big) \big( t^2 - c t - d \big) \\ &+ \big( \sigma^2 (c)
\sigma(c) c + \sigma^2 (d)c + \sigma^2 (c) \sigma(d) \big) t +
\sigma^2 (d)d + \sigma^2 (c) \sigma(c) d.
\end{split}
\end{equation}
If we define $$M_0 (c , d)(t) = 1, \ M_1 (c , d)(t) = t, \ M_2 (c ,
d)(t) = c t + d$$ $$M_3 (c , d)(t) = \big( \sigma(d) + \sigma(c)c
\big)t + \sigma(c) d,$$ $$M_4 (c ,d)(t) = \big( \sigma^2 (c)
\sigma(c) c + \sigma^2 (d)c + \sigma^2 (c) \sigma(d) \big) t +
\sigma^2 (d)d + \sigma^2 (c) \sigma(c) d,$$ then multiplying
\eqref{eqn:degree 4 t^2 identity 1}, \eqref{eqn:degree 4 t^2 identity
2} and \eqref{eqn:degree 4 t^2 identity 3} on the left by $a_i$ and
summing over $i$ yields $$f(t) = q(t) \big( t^2 - c t - d \big) + M_4
(c , d)(t) - \sum_{i=0}^{3} a_i M_i (c , d )(t)$$ for some $q(t) \in
R$. Thus the remainder after dividing $f(t)$ on the right by $t^2 - c
t - d$ is $$M_4 (c ,d )(t) - \sum_{i=0}^3 a_i M_i (c , d )(t),$$
which evidently implies:

\begin{proposition} \label{prop:degree 4 right divide by quadratic}
 $f(t) = t^4 - a_3 t^3 - a_2 t^2 - a_1 t - a_0 \in R = D[t;\sigma]$.
$(t^2 - c t - d) \vert_r f(t)$ is equivalent to $$\sigma^2 (c)
\sigma(c)c + \sigma^2 (d)c + \sigma^2 (c) \sigma(d) - a_3 \big(
\sigma(d) + \sigma(c)c \big) - a_2 c - a_1 = 0,$$ and $$\sigma^2 (d)d
+ \sigma^2 (c) \sigma(c) d - a_3 \sigma(c)d - a_2 d - a_0 = 0.$$
\end{proposition}

 Propositions
\ref{prop:leftdivdegree1} and \ref{prop:degree 4 right divide by
quadratic} together with  \eqref{prop:rightdivdegree1} yield:

\begin{theorem} \label{thm:degree 4 irreducibility criteria}
 $f(t) = t^4 - a_3 t^3
- a_2 t^2 - a_1 t - a_0 \in R = D[t;\sigma]$ is irreducible if and only if
\begin{equation} \label{eqn:degree 4 irreducible 1}
\sigma^3 (b) \sigma^2 (b) \sigma(b) b + a_3 \sigma^2 (b) \sigma(b) b
+ a_2 \sigma(b) b + a_1 b + a_0 \neq 0,
\end{equation}
and
\begin{equation} \label{eqn:degree 4 irreducible 2}
\begin{split}
&\sigma^3 (b) \sigma^2 (b) \sigma (b) b + \sigma^3 (b) \sigma^2 (b)
\sigma (b) a_3+ \\ &\sigma^3 (b) \sigma^2 (b) \sigma (a_2) + \sigma^3
(b) \sigma^2 (a_1) + \sigma^3 (a_0) \neq 0,
\end{split}
\end{equation}
for all $b \in D$, and for every $ c , d \in D$, we have
\begin{equation} \label{eqn:degree 4 irreducible 3}
\sigma^2 (c) \sigma(c)c + \sigma^2 (d)c + \sigma^2 (c) \sigma(d) +
a_3 (\sigma(d) + \sigma(c)c) + a_2 c + a_1 \neq 0,
\end{equation}
 or
\begin{equation} \label{eqn:degree 4 irreducible 4}
\sigma^2 (d)d + \sigma^2 (c) \sigma(c) d + a_3 \sigma(c)d + a_2 d +
a_0 \neq 0.
\end{equation}
\end{theorem}

I.e., $f(t)$ is irreducible if and only if \eqref{eqn:degree 4
irreducible 1} and \eqref{eqn:degree 4 irreducible 2} and (\eqref{eqn:degree 4 irreducible 3} or \eqref{eqn:degree 4
irreducible 4}) hold.

\begin{proof}
$f(t)$ is irreducible if and only if $(t-b) \nmid_r f(t)$ for all $b\in D$, $(t-b) \nmid_l f(t)$ for all $b \in D$ and $(t^2 - c t - d)
\nmid_r f(t)$ for all $c, d \in D$. Therefore the result follows from \eqref{prop:rightdivdegree1}, Propositions \ref{prop:leftdivdegree1}
and \ref{prop:degree 4 right divide by quadratic}.
\end{proof}

\begin{lemma} \label{lem:t^4-afactorisation1}
Let $f(t) = t^4 - a \in R$. Suppose $(t-b) \vert_r f(t)$, then $$f(t) = (t + \sigma^3(b))(t^2 + \sigma^2(b) \sigma(b))(t - b),$$ and $$f(t)
= (t^2 + \sigma^3(b) \sigma^2(b))(t + \sigma(b))(t-b),$$ are
factorisations of $f(t)$. In particular, $(t + \sigma(b))(t-b) = t^2 -\sigma(b)b$ also right divides $f(t)$.
\end{lemma}

\begin{proof}
Multiplying out these factorisations gives $t^4 - \sigma^3(b)
\sigma^2(b) \sigma(b) b$ which is equal to $f(t)$ by
\eqref{prop:rightdivdegree1}.
\end{proof}

Hence if $f(t) = t^4 - a$ has a right linear divisor then it also has
a right quadratic divisor and Theorem \ref{thm:degree 4
irreducibility criteria} simplifies to:

\begin{corollary} \label{thm:t^4-a irreducibility criteria delta=0}
$f(t) = t^4 - a \in R$ is reducible if and only if
$$\sigma^2 (c) \sigma(c)c + \sigma^2 (d)c + \sigma^2 (c) \sigma(d) =
0 \quad \text{and} \quad \sigma^2 (d)d + \sigma^2 (c) \sigma(c) d =
a,$$ for some $c, d \in D.$
\end{corollary}

\begin{proof}
By  Corollary \ref{cor:left iff right divisor}, $f(t)$ has a right
linear divisor if and only if it has a left linear divisor. Moreover
if $f(t)$ has a right linear divisor then it also has a quadratic
right divisor by Lemma \ref{lem:t^4-afactorisation1}, therefore
$f(t)$ is reducible if and only if $(t^2 - c t - d) \vert_r f(t)$ for
some $c, d \in D$. The result now follows from Proposition
\ref{prop:degree 4 right divide by quadratic}.
\end{proof}


\subsection{Examples in $\mathbb{F}_{p^h}[t;\sigma]$} 

Let $K = \mathbb{F}_{p^h}$ be a finite field of order $p^h$ for some prime $p$ and $\sigma$ be a non-trivial $\mathbb{F}_p$-automorphism
of $K$. This means $\sigma: K \rightarrow K, \ k \mapsto k^{p^r},$ for some $r \in \{ 1, \ldots, h-1 \}$. Here $\sigma$ has order $n = h/ \mathrm{gcd}(r,h)$ and ${\rm Gal}(K/\mathrm{Fix}(\sigma))=\langle\sigma\rangle$.
Algorithms for efficiently factorising polynomials in $\mathbb{F}_{p^h}[t;\sigma]$ exist, see \cite{G0} or more recently
\cite{caruso2017new}.

\begin{lemma} \label{lem:gcd number theory result}
$\mathrm{gcd}(p^h-1,p^r-1) = p^{\mathrm{gcd}(h,r)}-1.$
\end{lemma}

\begin{proof}
Let $d = \mathrm{gcd}(r,h)$ so that $h = dn$. We have $$p^{h}-1 =
(p^d-1)(p^{d(n-1)} + \ldots + p^d+1),$$ therefore $p^h-1$ is
divisible by $p^d-1$. A similar argument shows $(p^d-1) \vert
(p^r-1)$. Suppose that $c$ is a common divisor of $p^h-1$ and
$p^r-1$, this means $p^h \equiv p^r \equiv 1 \ \mathrm{mod} \ (c)$.
Write $d = hx + ry$ for some integers $x,y$, then we have $$p^d =
p^{hx + ry} = (p^h)^x (p^r)^y \equiv 1 \ \mathrm{mod} \ (c)$$ which
implies $c \vert (p^d-1)$ and hence $p^d-1 =
\mathrm{gcd}(p^h-1,p^r-1)$.
\end{proof}

Given $k \in K^{\times}$, we have $k \in \mathrm{Fix}(\sigma)$ if and only if $k^{p^r-1} = 1$, if and only if $k$ is a
$(p^r-1)^{\mathrm{th}}$ root of unity. There are
$\mathrm{gcd}(p^r-1,p^h-1)$ such roots of unity in $K$,
thus $$\vert\mathrm{Fix}(\sigma) \vert = \mathrm{gcd}(p^r-1,p^h-1) +
1 =p^{\mathrm{gcd}(r,h)}$$ by Lemma \ref{lem:gcd number theory
result} and so $\mathrm{Fix}(\sigma) \cong \mathbb{F}_{q}$ where $q
=p^{\mathrm{gcd}(r,h)}$. \label{page:Fix(sigma) isomorphic to}

\begin{proposition}
 (i) Suppose $n \in \{ 2,3 \}$, then $f(t) = t^n - a \in K[t;\sigma]$ is irreducible if and only if $a \in K \setminus
 \mathrm{Fix}(\sigma)$.
\\ (ii) Suppose $n$ is a prime and $n \vert (q-1)$. Then $f(t) = t^n - a \in K[t;\sigma]$ is irreducible if and only
if $a \in K \setminus \mathrm{Fix}(\sigma)$.
\\ In particular, in both (i) and (ii), there are precisely $p^{h} - q$ irreducible polynomials in
$K[t;\sigma]$ of the form $t^n-a$ for some $a \in K$.
\end{proposition}

\begin{proof}
 (i) $f(t)$ is irreducible if and only if
$\prod_{l=0}^{n-1} \sigma^l(b) = N_{K/\mathrm{Fix}(\sigma)}(b) \neq
a$ for all $b \in K$ by Theorem \ref{thm:Petit_factor} or Corollary
\ref{cor:Irreducibility t^3-a}, where $N_{K/\mathrm{Fix}(\sigma)}$ is
the field norm. It is well-known that as $K$ is a finite field,
$N_{K/\mathrm{Fix}(\sigma)}: K^{\times} \rightarrow
\mathrm{Fix}(\sigma)^{\times}$ is surjective and so $f(t)$ is
irreducible if and only if $a \notin \mathrm{Fix}(\sigma)$. There are
$p^{h} - q$ elements in $K \setminus \mathrm{Fix}(\sigma)$, hence
there are precisely $p^{h} - q$ irreducible polynomials of the form
$t^n-a$ for some $a \in K$.
\\ (ii)  $\mathrm{Fix}(\sigma) \cong \mathbb{F}_{q}$ contains a primitive $n^{\mathrm{th}}$ root of unity because
$n \vert (q-1)$ \cite[Proposition II.2.1]{koblitz1994course}. The rest of the proof is similar to (i) but uses Theorem
\ref{thm:Petit(19)}.
\end{proof}

Let $a, b \in K$ and recall $(t-b) \vert_r (t^m-a)$ is equivalent to $a = \sigma^{m-1}(b) \cdots \sigma(b)b = b^s$ by
\eqref{prop:rightdivdegree1} where $s = \sum_{j=0}^{m-1}p^{rj} =(p^{mr}-1)/(p^r-1)$. Suppose $z$ generates the multiplicative group
$K^{\times}$. Writing $b = z^l$ for some $l \in \mathbb{Z}$ yields
$(t-b) \vert_r(t^m-a)$ if and only if $a = z^{ls}$. This implies the following:

\begin{proposition} \label{prop:finite fields irreducibility primitive element}
Let $f(t) = t^m-a \in K[t;\sigma]$ and write $a \in K$ as $a = z^u$
for some $u \in \{ 0, \ldots, p^h-2 \}$.
\\ (i) $(t-b) \nmid_r f(t)$ for all $b \in K$ if and only if
 $u \notin \mathbb{Z} s \ \mathrm{mod} \ (p^h-1).$
\\ (ii) If $m \in \{2, 3\}$ then $f(t)$ is irreducible if and only if
 $u \notin \mathbb{Z} s \ \mathrm{mod} \ (p^h-1).$
\\ (iii) Suppose $m$ is a prime divisor of $(q-1)$, then $f(t)$ is irreducible if and only if  $u \notin \mathbb{Z} s \ \mathrm{mod} \ (p^h-1).$
\end{proposition}

\begin{proof}
 (i) $(t-b) \nmid_r f(t)$ for all $b \in K$ if and only if $a = z^u \neq z^{l s}$ for all $l \in \mathbb{Z}$, if and only if $u \notin \mathbb{Z} s \ \mathrm{mod} \ (p^h-1)$.
\\ (ii) $f(t)$ has a left linear divisor if and only if it has a right linear divisor by Corollary \ref{cor:left iff right divisor}. Therefore if $m \in \{2, 3 \}$ then $f(t)$ is irreducible if and only if $(t-b) \nmid_r f(t)$ for all $b \in K$ and so the assertion follows by (i).
\\ (iii) If $m$ is a prime divisor of $(q-1)$ then $\mathrm{Fix}(\sigma) \cong \mathbb{F}_{q}$ contains a primitive $m^{\mathrm{th}}$ root of unity. Therefore the result follows by (i) and Theorem \ref{thm:Petit(19)}.
\end{proof}

\begin{corollary} \label{cor:Finite field t^m-a irreducibility criteria}
 (i) There exists $a \in K$ such that $(t-b) \nmid_r (t^m-a)$ for all $b \in K$ if and only if $\mathrm{gcd}(s,p^h-1) > 1.$
\\ (ii) \cite[(22)]{P66} Suppose $m \in \{ 2,3 \}$ or $m$ is a prime divisor of $(q-1)$.
 Then there exists $a \in K^{\times}$ such that $t^m-a \in K[t;\sigma]$ is irreducible if and only if
 $\mathrm{gcd}(s,p^h-1) > 1.$
\end{corollary}

\begin{proof}
There exists $u \in \{ 0, \ldots, p^h-2 \}$ such that $u \notin
\mathbb{Z}s \ \mathrm{mod} \ (p^h-1)$, if and only if $s$ does not
generate $\mathbb{Z}_{p^h-1}$, if and only if $\mathrm{gcd}(s,p^h-1)
> 1.$ Hence the result follows by Proposition \ref{prop:finite fields
irreducibility primitive element}.
\end{proof}

\begin{corollary} \label{cor:p=1mod m finite field irreducibility criteria}
Suppose $p \equiv 1 \ \mathrm{mod} \ m$.
\\ (i) There exists $a \in K$ such that $(t-b) \nmid_r (t^m-a)$ for all $b \in K$.
\\ (ii) If $p$ is an odd prime, then there exists $a \in K^{\times}$ such that $t^2-a \in K[t;\sigma]$ is irreducible.
\\ (iii) If $m = 3$, then there exists $a \in K^{\times}$ such that $t^3-a \in K[t;\sigma]$ is irreducible.
\\ (iv) Suppose $m$ is a prime divisor of $(q-1)$, then there exists $a \in K^{\times}$ such that $t^m-a \in K[t;\sigma]$ is irreducible.
\end{corollary}

\begin{proof}
We have $$s \ \mathrm{mod} \ m = \sum_{i=0}^{m-1} (p^{ri} \
\mathrm{mod} \ m) \ \mathrm{mod} \ m = (\sum_{i=0}^{m-1} 1) \
\mathrm{mod} \ m = 0,$$ and $p^h \equiv 1 \ \mathrm{mod} \ m$. This
means $m \vert s$ and $m \vert (p^h-1)$, therefore $\mathrm{gcd}(
s,p^h-1) \geq m$ and so the assertion follows by Corollary
\ref{cor:Finite field t^m-a irreducibility criteria}.
\end{proof}


\section{Irreducibility criteria for polynomials of degree two and three and for $t^m-a$ in $D[t;\sigma,\delta]$}
\label{section:Irreducibility Criteria in D[t;sigma,delta]}

In this Section we generalize  some results from Section
\ref{section:Irreducibility Criteria in D[t;sigma]} to polynomials in
$R = D[t;\sigma,\delta]$, where $D$ is a division ring with center
$F$ and $\sigma$ an endomorphism of $D$. We recursively define a sequence of maps $N_i : D \rightarrow D, \ i \geq 0$,
 by $$N_0(b) = 1, \ N_{i+1}(b) = \sigma(N_i(b))b + \delta(N_i(b)),$$ i.e., $ N_1(b) = b, \ N_2(b) = \sigma(b)b + \delta(b),
\ldots$

Let $f(t) = t^m - \sum_{i=0}^{m-1} a_i t^i \in R$. Then $(t-b) \vert_r f(t)$ is equivalent to
\begin{equation} \label{prop:rightlinearfactordelta}
N_m(b) - \sum_{i=0}^{m-1} a_i N_i (b) = 0
\end{equation}
\cite[Lemma 2.4]{lam1988vandermonde}.

If $\sigma$ is an automorphism of $D$, we can also view $R=D[t;\sigma,\delta]$ as a right
polynomial ring. In particular, this means we can write $f(t) = t^m -\sum_{i=0}^{m-1} a_i t^i \in R$ in the form $f(t) = t^m -
\sum_{i=0}^{m-1} t^i a_i'$ for some uniquely determined $a_i' \in D$. To find the remainder after left division of $f(t)$
by $(t-b)$, we  recursively define a sequence
of maps $M_i:D \rightarrow D$, $i \geq 0$, by
$$M_{i+1}(b) = b \sigma^{-1}(M_i(b)) - \delta(\sigma^{-1}(M_i(b))), \
M_0(b) = 1,$$ that is $M_0(b) = 1$, \ $M_1(b) = b$, \ $M_2(b) = b
\sigma^{-1}(b) - \delta(\sigma^{-1}(b)), \ldots$

\begin{proposition} \label{prop:leftlinearfactordelta}
Suppose $\sigma$ is an automorphism of $D$. Then
$(t-b) \vert_l f(t)$ is equivalent to $M_m(b) - \sum_{i=0}^{m-1}
M_i(b) a_i' = 0$. In particular, $(t-b) \vert_l (t^m - a)$ if and
only if $M_m(b) \neq a$.
\end{proposition}

\begin{proof}
We first show $t^n - M_n(b) \in (t-b)R$ for all $b \in D$ and $n \geq
0$: If $n = 0$ then $t^0 - M_0(b) = 1 - 1 = 0 \in (t-b)R$ as
required. Suppose inductively $t^n - M_n(b) \in (t-b)R$ for some $n
\geq 0$, then
\begin{align*}
&t^{n+1} - M_{n+1}(b) = t^{n+1} - b \sigma^{-1}(M_n(b)) +
\delta(\sigma^{-1}(M_n(b))) \\ &= t^{n+1} + (t-b) \sigma^{-1}(M_n(b))
- t \sigma^{-1}(M_n(b)) + \delta(\sigma^{-1}(M_n(b))) \\ &= t^{n+1} +
(t-b) \sigma^{-1}(M_n(b)) - M_n(b) t - \delta(\sigma^{-1}(M_n(b))) +
\delta(\sigma^{-1}(M_n(b))) \\ &= (t-b) \sigma^{-1}(M_n(b)) + (t^n -
M_n(b))t \in (t-b)R,
\end{align*}
as $t^n-M_n(b) \in (t-b)R$. Therefore $t^n - M_n(b) \in (t-b)R$ for
all $b \in D$, $n \geq 0$ by induction.

As a result, there exists $q_i(t) \in R$ such that $t^i = (t-b)
q_i(t) + M_i(b)$, for all $i \in \{ 0, \ldots, m \}$. Multiplying on
the right by $a_i'$ and summing over $i$ yields $$f(t) = (t-b)q(t) +
M_m(b) - \sum_{i=0}^{m-1} M_i(b) a_i',$$ for some $q(t) \in R$.
\end{proof}

\begin{theorem} \label{thm:irredcriteriadelta}
 (i) $f(t) = t^2 - a_1 t - a_0 \in D[t;\sigma,\delta]$ is irreducible if and only if  $\sigma(b)b + \delta(b) - a_1 b - a_0 \neq 0$  for all $b \in D$.
\\ (ii) Suppose $\sigma$ is an automorphism of $D$ and $f(t) = t^3 - a_2 t^2 - a_1 t - a_0 \in D[t;\sigma,\delta]$.
Write $f(t) = t^3 - t^2 a_2' - t a_1' - a_0'$ for some unique $a_0', a_1', a_2' \in D$. Then $f(t)$ is irreducible if and only if
\begin{equation} \label{eqn:irredcriteriadelta 1}
N_3(b) - \sum_{i=0}^{2}a_i N_i(b) \neq 0,
\end{equation}
and
\begin{equation} \label{eqn:irredcriteriadelta 2}
M_3(b) - \sum_{i=0}^{2} M_i(b) a_i' \neq 0,
\end{equation}
for all $b \in D$.
\end{theorem}

\begin{proof}
 (i)  $f(t)$ is irreducible if and only if it has no right linear factors, if and only if
$$N_2 (b) - a_1 N_1(b) - a_0 N_0 (b) = \sigma(b)b + \delta(b) - a_1 b
- a_0 \neq 0,$$ for all $b \in D$ by
\eqref{prop:rightlinearfactordelta}.
\\ (ii)  $f(t)$ is irreducible if and only if it has no left or right linear factors, if and only if \eqref{eqn:irredcriteriadelta 1}
and \eqref{eqn:irredcriteriadelta 2} hold for all $b \in D$ by \eqref{prop:rightlinearfactordelta} and Proposition
\ref{prop:leftlinearfactordelta}.

\end{proof}

We can thus prove a generalization of Theorem \ref{thm:Petit(19)}:

\begin{theorem} \label{thm:Petit(19), delta not 0}
Suppose $m$ is prime, $\mathrm{char}(D) \neq m$ and $F \cap \mathrm{Fix}(\sigma)$ contains
a primitive $m$th root of unity $\omega$. Then $f(t) = t^m - a \in
D[t;\sigma,\delta]$ is irreducible if and only if $N_m(b) \neq a$ for
all $b \in D$.
\end{theorem}

\begin{proof}
Recall that
$$\delta(b^n) = \sum_{i=0}^{n-1} \sigma(b)^i
\delta(b)b^{n-1-i}$$ for all $b \in D$, $n \geq 1$ by \cite[Lemma
1.1]{Goodearl} and so
\begin{align*}
0 &= \delta(1) = \delta(\omega^m) = \sum_{i=0}^{m-1} \sigma(\omega)^i
\delta(\omega) \omega^{m-1-i} = \sum_{i=0}^{m-1} \omega^i
\delta(\omega) \omega^{m-1-i} \\ &= \sum_{i=0}^{m-1} \delta(\omega)
\omega^{m-1} = \delta(\omega) \omega^{m-1} m,
\end{align*}
where we have used $\omega \in F \cap \mathrm{Fix}(\sigma)$.
Therefore $\omega \in \mathrm{Const}(\delta)$ because
$\mathrm{char}(D) \neq m$, hence also $\omega^i \in
\mathrm{Const}(\delta)$ and so $(\omega t)^i = \omega^i t^i$ for all
$i \in \{ 1, \ldots , m \}$. Furthermore if $b \in D$, then $(t-b)
\nmid_r f(t)$ is equivalent to $N_m(b) \neq a$ by
\eqref{prop:rightlinearfactordelta}. The proof now follows exactly as
in Theorem
\ref{thm:Petit_factor}.
\end{proof}

\begin{corollary} \label{cor:Petit(19), delta not 0, m=3}
Suppose $\mathrm{char}(D) \neq 3$, $\sigma = id$ and $F \cap \mathrm{Fix}(\sigma)$
contains a primitive third root of unity. Then $f(t) = t^3-
a \in D[t; \delta]$ is irreducible if and only if $$N_3(b) = b^3 + 2
\delta(b)b + b\delta(b) + \delta^2(b) \neq a,$$ for all $b \in D$.
\end{corollary}


\end{document}